\documentclass[10pt]{article}
\usepackage{amsmath,amssymb,amsthm,graphics,graphicx}
\usepackage[T1]{fontenc}
\usepackage{multicol}
\usepackage{epstopdf}
\usepackage{relsize}
\usepackage{mathtools}

\usepackage{url}
\DeclareMathSizes{24}{40}{32}{24}

\newtheorem{theorem}{Theorem}

\newtheorem{lemma}[theorem]{Lemma}
\begin{document}
\vspace{5pt}
\begin{center}
\vspace{2pt}
    {\Large{Polynomials and Second Order Linear Recurrences}
 }
\\
{Soumyabrata Pal \footnote{Fourth Year student at IIT Kharagpur}   \\}
{Shankar M. Venkatesan \footnote {Bangalore, India}}
 \end{center}
 \vspace{4pt}
\begin{abstract}
One of the most interesting results of the last century was the proof completed by Matijasevich that computably enumerable sets are precisely the diophantine sets [MRDP Theorem, 9], thus settling, based on previously developed machinery, Hilbert's question whether there exists a general algorithm for checking the solvability in integers of any diophantine equation.
Also diophantine representations of the set of prime numbers were exhibited [10] involving polynomials of high degree and many variables [11]. While it is easy to show that there does not exist a univariate polynomial that evaluates to only prime numbers (not to mention all the prime numbers, the two-variable case for primes is a well known open problem [11]), so far there do not exist any results on techniques that help prove the non-existence of  multivariate polynomials of certain low degrees for interesting sets (such as the primes), or for showing lower bounds on the degree-variable trade-off for such polynomials. In this paper we describe techniques to prove the nonexistence of polynomials in two variables for some simple generalizations of the Fibonacci sequence (explicit diophantine representation of Fibonacci numbers were known from Jones' polynomial whose positive values have the same range as that of Fibonacci numbers), and we believe similar techniques exist for the primes. In this paper we mainly show the following results: (1) using one of the many techniques known for solving the Pell's equation, namely the solution in an extended number system, we prove the existence and explicitly find the polynomials for the recurrences of the form $e(n)=ae(n-1)+e(n-2)$ with starting values of 0 and 1 in particular, and for any arbitrary starting values, in the process defining a concept of fundamental starting numbers, (2) we prove a few identities that seem to be quite interesting and useful, (3) we use these identities in a novel way to generate systems of equations of certain rank deficiency using which we disprove for the first time the existence of any polynomial in 2 variables for the generalized recurrence of the form $e(n)=ae(n-1)+be(n-2)$ (even though these are obviously computably enumerable and hence diophantine), (4) using a known Cassini modification, we prove a similar non-existence for three variables. Our work raises questions about what techniques are good for establishing non-existence or for proving lower bounds on the degree and on the number of variables for the diophantine representation of these as well as other interesting sets such as primes.
\end {abstract}
\section{Introduction} 
The Fibonacci sequence 0,1,1,2,3,5,8,.... given by Italian mathematician Leonardo of Pisano has continued to produce a lot of research interest over the years. Things took a turn for the different when Matijasevich [2,3] first derived a fibonacci identity to complete the proof that computably enumerable sets are precisely the diophantine sets [MRDP Theorem, 9] thus finally settling Hilbert's tenth problem in the negative. He also showed how to construct a polynomial for the primes.
Polynomials were also exhibited to find the explicit formula for the n'th member of some sequences as well
as defining and working with Fibonacci polynomials and q-polynomials [6,7] . James P.Jones in 1975 [5] showed a very interesting result regarding the Fibonacci sequence: he showed
a polynomial namely $y(2-(x^{2}+xy-y^{2})^{2})$ whose positive values have the exact same range as the members of the Fibonacci series, 
and Jones et.al also showed a similar work for prime numbers [4]. In this paper the initial ideas are mainly influenced by Jones' work [5].
This paper also involves the well-known Pell's equation. Although Jones did not use this approach for his proof (he gave 
a completely different proof based on induction and inequalities) we have used a technique for solving Pell's equation [8], particularly one using factorization in a certain
extended number system (we believe similar approaches and factorizability exist in certain splitting fields). The result of Jones has never been extended for a generalized second order linear recurrences of the form $e(n)=ae(n-1)+be(n-2)$ :
In this paper we exhibit polynomials for recurrences of the form $e(n)=ae(n-1)+e(n-2)$ i.e b=1 and prove a few identities as well. We also extend
the formulation of these polynomials to sequences with arbitrary starting values. Last but not least, we also prove for the first time the non-existence of certain degree-2 or 2-variable 
polynomials for a general second order linear recurrence with $a>1$ and $b>1$, and of certain 3-variable polynomials using an extended Casssini's identity [12]. We have left unanswered but plan to investigate the highly interesting questions of bounds on the lowest degree or the least variable polynomials (which do exist) for our sets as well as for primes and similar interesting sets. 
\section{Analysis for the sequence where b=1}
\subsection{Recurrence relation}
Even for the relation $e(n)=ae(n-1)+be(n-2)$, the following recurrence is satisfied
\begin{theorem}
$$\text{For any positive integer i, } e(i)^{2}+ae(i+1)e(i)-{e(i+1)}^{2}=(-1)^{i+1} $$
\end{theorem}
\begin{proof}
The proof is by induction. According to our asssumption, $e(0)=0$, $e(1)=1$ and $e(2)=a$. So for $i=1$, we can observe that the theorem holds. Let us assume that 
the theorem holds for all $i<=k$. So for $i=k+1$ we have,
$$
\begin{array}{cl}
&e(k+1)^{2}+ae(k+2)e(k+1)-{e(k+2)}^{2}\\
=&e(k+1)^{2}+a(ae(k+1)+e(k))e(k+1)-{ae(k+1)+e(k)}^{2}\\
=&e(k+1)^{2}-ae(k+1)e(k)-e(k)^{2}\\
=&-(-e(k+1)^{2}+ae(k+1)e(k)+e(k)^{2})\\
=&(-1)^{k+2}
\end{array}
$$
Hence the theorem holds for $i=k+1$ and therefore the theorem is satisfied for all $i$.
\end{proof}
So $e(i)$ and $e(i+1)$ forms a pair of solutions (x,y) for the equation
$$ x^{2}+axy-y^{2}=\pm{1}$$ similar to the fibonacci sequence case where a=1. The question is whether the set (e(i),e(i+1)) forms the only solutions of the
given equation. We will show that it indeed does.
\begin{theorem}
The pair $(x,y) \rightarrow (e(2n-1),e(2n))$ forms the only solutions of the equation $x^{2}+axy-y^{2}=1$ 
\end{theorem}
\begin{proof}
To prove this we need to solve for the integer solutions of the equation  $Q(x,y)=x^{2}+axy-y^{2}=\pm{1}$ and show that they are equivalent to (e(2n-1),e(2n)) for some i.
So we can write Q(x,y) as the norm of the element $x+y(\dfrac{a+\sqrt{a^{2}+4}}{2})$ in the extended number system 
$$\mathbb{Z}[\dfrac{a+\sqrt{a^{2}+4}}{2}]:=[x+y(\dfrac{a+\sqrt{a^{2}+4}}{2}):x,y \in \mathbb{Z}] $$
Since $x=1,y=a$ forms the initial fundamental solution because $1$ is the minimum value of $x$, the solutions to the equation are given by the integer coefficients 
of the expression
$$(1+a(\dfrac{a+\sqrt{a^{2}+4}}{2}))^{n} \text{ for all positive integers n}  $$
Now we will go on to prove a lemma which will in turn also prove the theorem.
\begin{lemma}
$$(1+a(\dfrac{a+\sqrt{a^{2}+4}}{2}))^{n}=e(2n-1)+e(2n)\dfrac{a+\sqrt{a^{2}+4}}{2} $$ 
\end{lemma}
\begin{proof}
We will prove this again by induction. Obviously it holds for the case $n=1$. Now say that it holds for all $n<=k$. Now for $n=k+1$
$$
\begin{array}{cl}
 &(1+a(\dfrac{a+\sqrt{a^{2}+4}}{2}))^{k+1}\\\\
 =&(e(2k-1)+e(2k)\dfrac{a+\sqrt{a^{2}+4}}{2})(1+a(\dfrac{a+\sqrt{a^{2}+4}}{2}))\\\\
 =&\dfrac{a+\sqrt{a^{2}+4}}{2}(ae(2k-1)+e(2k)(a^{2}+1))+(e(2k-1)+ae(2k))\\\\
 =&e(2k+1)+\dfrac{a+\sqrt{a^{2}+4}}{2}(a(e(2k-1)+ae(2k))+e(2k))\\\\
 =&e(2k+1)+\dfrac{a+\sqrt{a^{2}+4}}{2}(ae(2k+1)+e(2k))\\\\
 =&e(2k+1)+\dfrac{a+\sqrt{a^{2}+4}}{2}e(2k+2)
\end{array}
$$
\end{proof}
This shows that the lemma holds and the only solutions to this problem is the members of the series (e(2n-1),e(2n)). As a result of this theorem the following theorem
also holds.
\begin{theorem}
 $(x,y) \rightarrow (e(2n),e(2n+1))$ forms all the solutions for the equation $x^{2}+axy-y^{2}=-1$
\end{theorem}
\end{proof}
\begin{theorem}
 The set of all  members of the generalized sequence e(n)=ae(n-1)+e(n-2) is identical with the set of positive values of the polynomial
 $$y(2-(x^{2}+axy-y^{2})^{2})$$
\end{theorem}
\begin{proof}
 For x and y as members of the sequence the value of the polynomial boils down to y which is a positive member of the sequence. To see the converse, assume that 
 x and y are positive integers for which the value of the polynomial is positive. In that case we have,
 $$0<(x^{2}+axy-y^{2})^{2}<2$$ which means that
 $$(x^{2}+axy-y^{2})^{2}=1$$ guranteeing that x and y are members of the sequence. Hence the theorem is proved
\end{proof}
\subsection{An interesting identity}
In the previous subsection we observed that the solution of the equation $y^{2}-axy-x^{2}=1$ are members of the sequence $e(n)=ae(n-1)+e(n-2)$ where $(x,y) \rightarrow
(e(2k),e(2k+1))$. So if y has to be an integer for some x then the discriminant of the quadratic equation in y must be a square. We find,
$$\Delta=a^{2}x^{2}+4(x^{2}+1)=p^{2}$$ for some integer p.
Then we have
$$
\begin{array}{cl}
& p^{2}-x^{2}(a^{2}+4)=4\\
or&(\dfrac{p}{2})^{2}-(\dfrac{x}{2})^{2}(a^{2}+4)=1\\
\end{array}
$$
The equation turns out to be a simple Pell's equation. The solution in the extended number system 
$$\mathbb{Z}[\sqrt{a^{2}+4}]:=[x+y\sqrt{a^{2}+4}:2x,2y \in \mathbb{Z}] $$
is simply
$$\dfrac{p(i)}{2}+\dfrac{e(2i)\sqrt{a^{2}+4}}{2}$$ where p(i) is the corresponding square for the discriminant.
The solution to Pell's equation in this extended number system follows multiplicity and hence when we multiply 2 solutions of this system and collect the coefficient of
$\sqrt{a^{2}+4}$ we find it to be 
$$\dfrac{p(i)e(2j)}{4}+\dfrac{p(j)e(2i)}{4}$$
which must be half of another member of the sequence $e$ and specifically must be $\dfrac{e(2(i+j))}{2}$ since product on this integer system means addition on the index. Now replacing the
values of p(i) and p(j) we get the following theorem
\begin{theorem}
 $$e(2(i+j))=\dfrac{e(2j)}{2}\sqrt{e(2i)^{2}(a^{2}+4)+4}+\dfrac{e(2i)}{2}\sqrt{e(2j)^{2}(a^{2}+4)+4}$$
\end{theorem}
For the fibonacci sequence this simply turns out to be
$$f(2(i+j))=\dfrac{f(2j)}{2}\sqrt{5f(2i)^{2}+4}+\dfrac{f(2i)}{2}\sqrt{5f(2j)^{2}+4}$$
Hence this also gives a very interesting way to check whether a given number $x$ is an even-indexed member of the sequence $e(n)=ae(n-1)+be(n-2)$. It is so if 
$x^{2}(a^{2}+4)+4$ is a perfect square.\\
Similarly we can draw up conclusions for the odd indexed members of the sequence. The quadratic equation can be rewritten as $x^{2}+axy-y^{2}+1=0$ where the discriminant
must be a square for an integer solution.
$$\Delta=a^{2}y^{2}+4(y^{2}-1)=p^{2}$$ for some integer p.
Then we have
$$
\begin{array}{cl}
& p^{2}-y^{2}(a^{2}+4)=-4\\
or&(\dfrac{y}{2})^{2}(a^{2}+4)-(\dfrac{p}{2})^{2}=1\\
\end{array}
$$
The solution in the extended number system 
$$\mathbb{Z}[\sqrt{a^{2}+4}]:=[x\sqrt{a^{2}+4}+y:2x,2y \in \mathbb{Z}] $$
is simply
$$\dfrac{p(i)}{2}+\dfrac{e(2i+1)\sqrt{a^{2}+4}}{2}$$ where p(i) is the corresponding square for the discriminant. We observe that this is just the same as the case
of the even-indexed members. Hence again the product of the 2 solutions e(2i+1) and e(2j+1) will lead to an addition over the index but this results in an even-indexed 
member of the sequence. This is because in both cases the number system and the solution form turns out to be the same.
\begin{theorem}
 $$e(2(i+j)+2)=\dfrac{e(2j+1)}{2}\sqrt{e(2i+1)^{2}(a^{2}+4)-4}+\dfrac{e(2i+1)}{2}\sqrt{e(2j+1)^{2}(a^{2}+4)-4}$$
\end{theorem}
For the fibonacci sequence this simply turns out to be
$$f(2(i+j)+2)=\dfrac{f(2j+1)}{2}\sqrt{5f(2i+1)^{2}-4}+\dfrac{f(2i+1)}{2}\sqrt{5f(2j+1)^{2}-4}$$
and also
\begin{theorem}
 $$e(2(i+j)+1)=\dfrac{e(2j)}{2}\sqrt{e(2i+1)^{2}(a^{2}+4)-4}+\dfrac{e(2i+1)}{2}\sqrt{e(2j)^{2}(a^{2}+4)+4}$$
\end{theorem}
For the fibonacci sequence this simply turns out to be
$$f(2(i+j)+1)=\dfrac{f(2j)}{2}\sqrt{5f(2i+1)^{2}-4}+\dfrac{f(2i+1)}{2}\sqrt{5f(2j)^{2}+4}$$
Again this gives us a way of checking whether a number x is an odd-indexed member of the sequence $e(n)=ae(n-1)+e(n-2)$. We simply have to check whether
$x^{2}(a^{2}+4)-4$ is a perfect square or not.
\subsection{Sequences with starting numbers p and q}
In the previous section we have analysed sequences of the form $e(n)=ae(n-1)+e(n-2)$ where the starting values are 0 and 1 respectively. We will now analyse the 
case where the starting values are p and q (let us denote the corresponding sequence by $g(.)$)
We can always backtrace the previous values of the sequence by using $q-pa$ so that we obtain $\hat{p}$ and $\hat{q}$ such
that $\hat{q}-\hat{p}a<0$. Let us call these the fundamental starting values. So we can always obtain these fundamental starting values from the given starting values
for such a sequence. So without loss of generality we can assume the fundamental values of the sequence to be $p$ and $q$ i.e. $g(0)=p,g(1)=q,g(2)=p+qa$ such that 
$q-pa<0$.\\
The recurrence relation in this case will be
\begin{lemma}
$$\text{For any positive integer i, } g(i)^{2}+ag(i+1)g(i)-{g(i+1)}^{2}=(-1)^{i+1}R \text{ for some constant R}$$
\end{lemma}
The proof remains the same while R can be found by plugging in $i=0$ and we have
$$R=q^{2}-apq-p^{2}$$
Now let us consider a similar polynomial
$$Q(x,y)=x^{2}+axy-y^{2}=R$$
and we know that the solution is $(x,y) \rightarrow (e(2n-1),e(2n))$. Here again Q(x,y) is the norm of the element  $x+y(\dfrac{a+\sqrt{a^{2}+4}}{2})$ in the extended number system 
$$\mathbb{Z}[\dfrac{a+\sqrt{a^{2}+4}}{2}]:=[x+y(\dfrac{a+\sqrt{a^{2}+4}}{2}):x,y \in \mathbb{Z}] $$ but this time the norm of the solution is $\sqrt{R}$. So we get the
solutions to this equation by repeatedly multiplying the fundamental solution in $\mathbb{Z}[\dfrac{a+\sqrt{a^{2}+4}}{2}]$ to the equation $$Q(x,y)=x^{2}+axy-y^{2}=1$$ but we already
know it to be $$1+a(\dfrac{a+\sqrt{a^{2}+4}}{2})$$. Also the fundamental solution to the equation $Q(x,y)=x^{2}+axy-y^{2}=R$ must be
$$q+(qa+p)(\dfrac{a+\sqrt{a^{2}+4}}{2})$$ since $(q,p+qa) \rightarrow (g(1),g(2))$. So this is exactly similar to rotation in a circle of Radius>1 in complex numbers.
Hence we must have 
\begin{theorem} 
$$(q+(qa+p)(\dfrac{a+\sqrt{a^{2}+4}}{2}))(1+a(\dfrac{a+\sqrt{a^{2}+4}}{2}))^{n-1}=g(2n-1)+g(2n)(\dfrac{a+\sqrt{a^{2}+4}}{2})$$
\end{theorem}
This allows us to draw up relations between $e(.)$ and $g(.)$. We have, using Theorem 10 and Lemma 3
$$
\begin{array}{cl}
&(q+(qa+p)(\dfrac{a+\sqrt{a^{2}+4}}{2}))(e(2n-3)+e(2n-2)(\dfrac{a+\sqrt{a^{2}+4}}{2}))\\
=&qe(2n-3)+(\dfrac{a+\sqrt{a^{2}+4}}{2})(qe(2n-2)+(qa+p)e(2n-3))+\dfrac{(qa+p)e(2n-2)(a^{2}+2+a\sqrt{a^{2}+4})}{2}\\
=&(qe(2n-3)+(qa+p)e(2n-2))+(qe(2n-2)+(qa+p)e(2n-3)+a(qa+p)e(2n-2))(\dfrac{a+\sqrt{a^{2}+4}}{2})
\end{array}
$$
Comparing, we have
$$
\begin{array}{rcl}
g(2n-1)&=&qe(2n-3)+(qa+p)e(2n-2)\\
g(2n)&=&(q(a^{2}+1)+p)e(2n-2)+(qa+p)e(2n-3)
\end{array}
$$
\begin{theorem}
 The set of all  members of the generalized sequence e(n)=ae(n-1)+e(n-2) with fundamental starting values $p$ and $q$ is 
 identical with the set of positive integer values of the polynomial
 $$y(2-(\dfrac{x^{2}+axy-y^{2}}{R})^{2})$$ where $R=q^{2}-apq-p^{2}$
\end{theorem}
\begin{proof}
 For x and y as members of the sequence the value of the polynomial boils down to y which is a positive member of the sequence. To see the converse, assume that 
 x and y are positive integers for which the value of the polynomial is a positive integer. In that case we must have,
 $$(x^{2}+axy-y^{2})^{2}=R^{2}$$ guranteeing that x and y are members of the sequence. Hence the theorem is proved
 \end{proof}
\section{Analysis of sequences with $a,b>1$}
Let us consider the recurrence $e(n)=ae(n-1)+be(n-2)$ where $a,b \neq 0$ and are positive integers. A second order recurrence, if exists for this sequence must be of
the form 
$$f(a,b)(e(i+1))^{2}+g(a,b)e(i+1)e(i)+h(a,b)e(i)^{2}=(-1)^{i}$$
where f,g,h are functions of $a$ and $b$. So this recurrence has to hold for $i=1$ . On using that $e(0)=0$, $e(1)=1$ and $e(2)=a$, we have\\
$$a^{2}f+ag+h=-1$$
So assuming that the recurrence holds for $i<=k$, we put $i=k+1$ and we have,
$$
\begin{array}{cl}
 &f(e(k+2))^{2}+ge(k+2)e(k+1)+h(e(k+1))^{2}\\
 =&f(ae(k+1)+be(k))^{2}+ge(k+1)(ae(k+1)+be(k))+h(e(k+1))^{2}\\
 =&e(k+1)^{2}(af^{2}+ga+h)+e(k)e(k+1)(2fab+gb)+e(k)^{2}fb^{2}\\
\end{array}
$$
This must be equal to
$$-(f(e(k+1))^{2}+ge(k+1)e(k)+h(e(k))^{2})$$
Equating the terms and solving the 4 equations, we have
$$
\begin{array}{lcl}
f&=&1\\
g&=&-\dfrac{2ab}{b+1}\\
h&=&-b^{2}
\end{array}
$$
However on substituting the values of $f,g,h$ in the equation $a^{2}f+ag+h=-1$ gives us a condition on $a$ and $b$.
$$a^{2}+1-b^{2}=\dfrac{2a^{2}b}{b+1}$$
Since b>1 according to our assumption,
$$a^{2}+1-b^{2}<a^{2}$$
Then,
$$\begin{array}{lcl}
\dfrac{2a^{2}b}{b+1}&<&a^{2}\\\\
\text{or }\dfrac{b+1}{b}&>&2\\\\
\text{or }b<1
\end{array}
$$
which is a contradiction to our assumption. Hence,
\begin{lemma}
There does not exist any such second order recurrence for $e(n)=ae(n-1)+be(n-2)$ and hence a second order
polynomial cannot exist for which the analysis of the previous sections can hold.
\end{lemma}
Now the question remains whether there exists higher degree polynomials for such recurrences.
\section{Search for higher degree polynomials}
So the question posed in the previous section is the existence of higher degree polynomials whose solution or subset of solutions is the members of the sequence
$e(n)=ae(n-1)+be(n-2)$. So this polynomial must be a function of $a$ and $b$ and hence let us put $a,b=1$ and search for these kind of polynomials for fibonacci sequence
itself. So let us take the assumption that there exists a recurrence in the degree $n$ for the fibonacci sequence.
Let us define the expression 
$$X(i)=a_{0}(f(i+1))^{n}+a_{1}(f(i+1))^{n-1}f(i)+a_{2}(f(i+1))^{n-2}f(i)^{2}+\cdots+a_{n}(f(i))^{n}=(-1)^{i}C \text{ for some positive integer C}$$ So if this holds
then $X(i+1)=-X(i)$ must hold. We can write
$$X(i+1)=a_{0}(f(i+2))^{n}+a_{1}(f(i+2))^{n-1}f(i+1)+a_{2}(f(i+2))^{n-2}f(i+1)^{2}+\cdots+a_{n}(f(i+1))^{n}$$
Now if we use $f(i+2)=f(i+1)+f(i)$ and collecting the terms, we must have\\
$$C((f(i+1))^{n-r}f(i)^{r})=a_{0}\binom{n}{r}+a_{1}\binom{n-1}{r}+\cdots+a_{n-r}\binom{r}{r}$$
where C(.) represents the coefficient of that particular term. Because of the required recurrence on X, we must have
$$C((f(i+1))^{n-r}f(i)^{r})=-a_{r}$$
So, for a positive integer n, we have n+1 variables and therefore n+1 equations. Hence the resultant system can be written in matrix form
$$
\begin{bmatrix}
 2 & 1 & 1 & \dots & 1\\
 \binom{n}{1} & 1+\binom{n-1}{1} & \binom {n-2}{1} & \dots & 0\\
 \binom{n}{2} & \binom{n-1}{2} & \binom{n-2}{2}+1\dots & 0 & 0\\
 \vdots & \vdots & \vdots & \vdots & \vdots\\
 \binom{n}{n} & 0 & 0 &\dots & 1
\end{bmatrix}
\begin{bmatrix}
 a_{0} \\ a_{1} \\ a_{2} \\ \vdots \\ a_{n} 
\end{bmatrix}
=
\begin{bmatrix}
 0 \\ 0 \\ 0 \\ \vdots \\ 0 
\end{bmatrix}
$$
To give a few examples we have for $n=2$,
$$
\begin{array}{rcl}
 2a_{0}+a_{1}+a_{2}=0\\
 2a_{0}+2a_{1}=0\\
 a_{0}+a_{2}=0
\end{array}
$$
In this case the system of equations is not of full rank and hence we have $a_{0}=-a{1}=-a{2}$. Again for $n=3$, we have
$$
\begin{array}{rcl}
 2a_{0}+a_{1}+a_{2}+a_{3}=0\\
 3a_{0}+3a_{1}+a_{2}=0\\
 3a_{0}+a_{1}+a_{2}=0\\
 a_{0}+a_{3}=0
\end{array}
$$
In this case the system of equation is of full rank and we have $a_{0}=a_{1}=a_{2}=a_{3}=0$. So can we generalize this for any arbitrary n? From the nature of the 
equations we see that the matrix is almost row reduced i.e every row has a lesser number of variables than the previous one except 2 rows in the middle. To investigate,
we take up 2 cases.\\
\textit{Case I: n is odd}\\ In this case the 2 equations having same number of variable and also the same variables is for $r=\dfrac{n+1}{2}$ and $r=\dfrac{n-1}{2}$. We can write them as
$$
\begin{array}{rcl}
\binom{n}{\tfrac{n-1}{2}}a_{0}+\binom{n-1}{\tfrac{n-1}{2}}a_{1}+\cdots+a_{\tfrac{n+1}{2}}&=&0\\
\binom{n}{\tfrac{n+1}{2}}a_{0}+\binom{n-1}{\tfrac{n+1}{2}}a_{1}+\cdots+a_{\tfrac{n-1}{2}}+a_{\tfrac{n+1}{2}}&=&0
\end{array}
$$
Hence the only chance for the entire matrix to be indeterminate if these 2 equations are integer multiples of each other which they are clearly not. Hence there does
not exist such recurrences for fibonacci series for degree $n=odd$]\\
\textit{Case II: n is even}\\
The analysis is almost same as for $n=odd$ but in this casethere exists a system of 3 equations which is indeterminate iff the entire matrix is not of full rank. The
3 equations are for $r=\dfrac{n}{2}$,$r=\dfrac{n}{2}+1$ and $r=\dfrac{n}{2}-1$. We can write the 3 equations as\\
$$
\begin{array}{rcl}
\binom{n}{\tfrac{n}{2}-1}a_{0}+\binom{n-1}{\tfrac{n}{2}-1}a_{1}+\cdots+a_{\tfrac{n}{2}+1}=0\\
\binom{n}{\tfrac{n}{2}}a_{0}+\binom{n-1}{\tfrac{n}{2}}a_{1}+\cdots+2a_{\tfrac{n}{2}}=0\\
\binom{n}{\tfrac{n}{2}+1}a_{0}+\binom{n-1}{\tfrac{n}{2}+1}a_{1}+\cdots+a_{\tfrac{n}{2}-1}+a_{\tfrac{n}{2}+1}=0
\end{array}
$$
Again this system of equations does not have full rank only for $n=2$ and for any other $n$ all a's must be 0. Hence we can conclude that there does not exist any recurrence
of such kind for any $n>2$. Hence we can conclude that 
\begin{lemma}
There does not exist any polynomial with the properties discussed in the previous sections of any higher degree in 2 variables
for the fibonacci series and therefore for the general series. 
\end{lemma}
So we must analyse the existence of polynomials with 3 variables. Interestingly Vella and Vella showed in [12] the modified version of Cassini's identity for the generalized sequences.
$$f(i+2)f(i)-f(i+1)^{2}=(-b)^{i}(f(2)f(0)-f(1)^{2})$$
So replacing $f(i+2)=af(i+1)+bf(i)$ and $f(1)=1$, $f(0)=0$ we get 
$$bf(i)^{2}+af(i+1)f(i)-f(i+1)^{2}+(-b)^{i}=0$$
Therefore the next question to be answered is whether the polynomial $bx^{2}+axy-y^{2}+z=0$ has only triplets of the form $(f(i),f(i+1),(-b)^{i})$ as the only solutions.
Evidently it is not since for $a=b=1$ the triplet $(p,1,p(p+1)+1)$ also forms a solution for any integer $p$. So this is not an one to one function unlike the previous cases.
\section{Conclusions and Future work}
This paper attempts to answer difficult questions on the existence of lower degree or fewer variable polynomial representations of certain interesting sets such as the generalized Fibonacci numbers (where the positive range is either Z+ or injective into Z+).
We have described some new techniques in this process, and have come up with the concept of fundamental starting numbers. The future scope of work can be vast since similar investigations can be attempted for different kinds of recurrences, polynomials, or interesting sets such as the primes.  

\end{document}